\documentclass[12pt]{article}
\usepackage{amsmath, amsthm}\usepackage{enumerate}\usepackage{amssymb}
\usepackage{bbold}
\usepackage{bbm}\usepackage{dsfont}\usepackage{color}\usepackage{xcolor}
\usepackage[explicit]{titlesec}
\newtheorem{theorem}{Theorem}[subsection]
\newtheorem{proposition}[theorem]{Proposition}

\newtheorem{lemma}[theorem]{Lemma}
\newtheorem{example}[theorem]{Example}
\newtheorem{corollary}[theorem]{Corollary}

\newtheorem{assertion}[theorem]{Assertion}

\newtheorem{definition}[theorem]{Definition}

\makeatletter
\renewcommand{\subsection}{\@startsection{subsection}{1}
{0pt}{3.25ex plus 1ex minus.2ex}{-1em}{\normalfont\normalsize\bf}}
\makeatother
\begin{document}

\title{{Operators Affiliated to Banach Lattice Properties and the Enveloping Norms}}
\date{}
\maketitle\author{\centering{{Eduard Emelyanov, Svetlana Gorokhova\\ 
}

\bigskip
\bigskip
\abstract{Several recent papers were devoted to various 
modifications of limited, Gro\-then\-dieck, 
and Dunford--Pettis operators, etc., through involving the Banach lattice structure.
In the present paper, it is shown that many of these operators appear as
operators affiliated to well known properties of Banach lattices,
like the disjoint (dual) Schur property, the disjoint Gro\-then\-dieck property,
the property (d), and the sequential \text{\rm w}$^\ast$-continuity
of the lattice operations. It is proved that the spaces consisting of 
regularly versions of the above operators are all Banach spaces. 
The domination problem for such operators is investigated.}
\vspace{5mm}

{\bf Keywords:} Banach lattice, affiliated operators, enveloping norm, domination problem

{\bf MSC2020:} {\normalsize 46B25, 46B42, 46B50, 47B60}

}}

\section{Preliminaries}

Throughout the paper, vector spaces are real; operators are 
linear and boun\-ded; letters $X$, $Y$ stands for Banach spaces; 
and $E$, $F$ for Banach lattices. We denote by $B_X$ 
the closed unit ball of $X$; by $\text{\rm L}(X,Y)$ 
the space of all bounded operators 
from $X$ to $Y$; and by $E_+$ the positive cone of $E$. 
An operator $T:E\to F$ is called {\em regular} if $T=T_1-T_2$ for some 
$T_1,T_2\in\text{\rm L}_+(E,F)$. We denote by $\text{\rm L}_r(E,F)$ 
($\text{\rm L}_{ob}(E,F)$, $\text{\rm L}_{oc}(E,F)$) the space of all regular 
(\text{\rm o}-bounded, \text{\rm o}-continuous) operators from $E$ to $F$.

\subsection{}
Recall that a bounded $A\subseteq X$ is said to be a {\em limited} set 
(resp. a \text{\rm DP}-{\em set}) if each 
\text{\rm w}$^\ast$-null (resp. \text{\rm w}-null) sequence in $X'$ is 
uniformly null on $A$. Similarly, a bounded $A\subseteq E$ is called 
an \text{\rm a}-{\em limited} set (resp. an \text{\rm a-DP}-{\em set}) if each disjoint
\text{\rm w}$^\ast$-null (resp. disjoint \text{\rm w}-null) sequence in $E'$ is 
uniformly null on $A$ (cf. \cite{AE,AEW,CCJ,El}).
Each relatively compact set is limited, each limited set is an 
\text{\rm a}-limited \text{\rm DP}-set, and each  
\text{\rm DP}-set is an \text{a-\rm DP}-set.

\begin{assertion}{\em (cf. \cite{BD})}
{\em Let $A\subseteq X$ be limited.  Then$:$
\begin{enumerate}[{\rm (i)}]
\item 
Every sequence in $A$ has a \text{\rm w}-Cauchy subsequence.
\item 
If $X$ is either separable or else reflexive, then $A$ is relatively compact.
\item 
If $\ell^1$ does not embed in $X$, then $A$ is relatively \text{\rm w}-compact.
\end{enumerate}}
\end{assertion}
\noindent
The following technical fact (cf. \cite[Prop.1.2.1]{AEG4}) is useful.

\begin{assertion}\label{uniform convergence on A}
{\em Let $A\subseteq X$ and $B\subseteq X'$ be nonempty. Then$:$
\begin{enumerate}[{\rm (i)}] 
\item
A sequence $(f_n)$ in $X'$ is uniformly null on $A$
iff $f_n(a_n)\to 0$ for each sequence $(a_n)$ in $A$. 
\item
A sequence $(x_n)$ in $X$ is uniformly null on $B$
iff $b_n(x_n)\to 0$ for each sequence $(b_n)$ in $B$.
\end{enumerate}}
\end{assertion}
\noindent
A bounded $B\subseteq X'$ (resp. $B\subseteq E'$)
is called an \text{\rm L}-{\em set} 
(resp. an \text{\rm a-L}-{\em set})
if each \text{\rm w}-null sequence in $X$ 
(resp. each disjoint \text{\rm w}-null sequence in $E$)
is uniformly null on $B$ (cf. \cite{Ghenciu}). The next fact follows
from Assertion \ref{uniform convergence on A}.

\begin{assertion}\label{limited and DP sets}
{\em A bounded subset $A$ of $X$ is
\begin{enumerate}[{\rm (i)}]
\item
limited iff $f_n(a_n)\to 0$ for all 
\text{\rm w}$^\ast$-null $(f_n)$ in $X'$ and all $(a_n)$ in $A$;
\item
a \text{\rm DP}-set iff $f_n(a_n)\to 0$ for all  
\text{\rm w}-null $(f_n)$ in $X'$ and all $(a_n)$ in $A$.
\end{enumerate}
A bounded subset $B$ of $X'$ is
\begin{enumerate}
\item[{\rm (iii)}]
an \text{\rm L}-set iff $b_n(x_n)\to 0$ for all $(b_n)$ in $B$ and all 
\text{\rm w}-null $(x_n)$ in $X$.
\end{enumerate}
A bounded subset $A$ of $E$ is
\begin{enumerate}
\item[{\rm (iv)}]
\text{\rm a}-limited iff $f_n(a_n)\to 0$ for all 
disjoint \text{\rm w}$^\ast$-null $(f_n)$ in $E'$ and all $(a_n)$ in $A$;
\item[{\rm (v)}]
an \text{\rm a-DP}-set iff $f_n(a_n)\to 0$ for all  
disjoint \text{\rm w}-null $(f_n)$ in $E'$ and all $(a_n)$ in $A$.
\end{enumerate}
A bounded subset $B$ of $E'$ is
\begin{enumerate}
\item[{\rm (vi)}]
an \text{\rm a-L}-set
iff $b_n(x_n)\to 0$ for all $(b_n)$ in $B$ and all 
disjoint \text{\rm w}-null $(x_n)$ in $E$.
\end{enumerate}}
\end{assertion}

\subsection{}
Let us recall the following properties of Banach spaces 
and describe operators affiliated to these properties.

\begin{definition}\label{Main Schur property} 
{\em A Banach space $X$ is said to possess:
\begin{enumerate}[a)]
\item 
the {\em Schur property} (briefly, $X\in\text{\rm (SP)}$)
if each \text{\rm w}-null sequence in $X$ is norm null;
\item 
the {\em Grothendieck property} (briefly, $X\in\text{\rm (GP)}$) 
if each \text{\rm w}$^\ast$-null sequence in $X'$ is \text{\rm w}-null;
\item
the {\em Dunford--Pettis property}  (briefly, $X\in\text{\rm (DPP)}$) if $f_n(x_n)\to 0$ for 
each \text{\rm w}-null $(f_n)$ in $X'$ and each
\text{\rm w}-null $(x_n)$ in $X$;
\item 
the {\em Gelfand--Phillips property} (briefly, $X\in\text{\rm (GPP)}$)
if each limited subset of $X$ is relatively compact (cf. \cite[p.424]{Ghenciu}).
\item 
the {\em Bourgain--Diestel property} (briefly, $X\in(\text{\rm BDP})$)
if each limited subset of $X$ is relatively \text{\rm w}-compact \cite{Emma1}.
\end{enumerate}}
\end{definition}
\noindent
Dedekind complete AM-spaces with a strong order unit belong to \text{\rm (GP)},
for a comprehensive rescent source on the Grothendieck property see \cite{GK}.
All separable and all reflexive Banach spaces belong to (\text{\rm GPP}) \cite{BD}.
A Dedekind $\sigma$-complete Banach lattice $E$ belongs $(\text{\rm GPP})$ 
iff $E$ has $\text{\rm o}$-continuous norm \cite{Buh}.
In particular, $c_0,\ell^1\in(\text{\rm GPP})$, yet $\ell^\infty\not\in(\text{\rm GPP})$. 
Clearly, $\text{\rm (GPP)}\Rightarrow\text{\rm (BDP)}$. By \cite{BD},
$X\in\text{\rm (BDP)}$ whenever $X$ contains no copy of $\ell^1$. 
Applying redistribution (as in \cite{AEG2}) between the domain and range
to the properties of Definition \ref{Main Schur property},
we obtain the following list of the affiliated operators.

\begin{definition}\label{Main Schur affiliated operators} 
{\em An operator $T:X\to Y$ is called:
\begin{enumerate}[a)]
\item 
an \text{\rm [SP]}-{\em operator} if $(Tx_n)$ is norm null for each
\text{\rm w}-null $(x_n)$ in $X$;
\item 
a \text{\rm [GP]}-{\em operator} if $(T'f_n)$ is \text{\rm w}-null in $X'$
for each \text{\rm w}$^\ast$-null $(f_n)$ in $Y'$;
\item
a \text{\rm [DPP]}-{\em operator} if $f_n(Tx_n)\to 0$ for 
each \text{\rm w}-null $(f_n)$ in $Y'$ and each
\text{\rm w}-null $(x_n)$ in $X$;
\item 
a $\text{\rm [GPP]}$-{\em operator} if $T$ carries 
limited sets onto relatively compact sets;
\item 
a $\text{\rm [BDP]}$-{\em operator} if $T$ carries 
limited sets onto relatively \text{\rm w}-compact sets.
\end{enumerate}}
\end{definition}
\noindent 
Note that \text{\rm [SP]}-operators coincide with Dunford--Pettis operators,
\text{\rm [GP]}-opera\-tors coincide with Grothendieck operators,
whereas \text{\rm [DPP]}-operators agree with weak Dunford--Pettis operators
of \cite[p.349]{AlBu}.

\begin{definition}\label{Main affiliated property}
{\em Let ${\cal P}$ be a class of operators between Banach spaces.
A Banach space $X$ is said to be {\em affiliated with} ${\cal P}$ 
if $I_X\in{\cal P}$. In this case we write $X\in({\cal P})$.}
\end{definition}
\noindent
It should be clear that if $(P)$ is one of the five
properties mentioned in Definition \ref{Main Schur property},
then $X\in(P)$ iff $X$ affiliated with $[P]$-operators; symbolically
$([(P)])=(P)$. It is worth noticing that the reflexivity of Banach spaces is affiliated 
with \text{\rm w}-compact operators and vice versa, whereas the finite 
dimensionality is affiliated with compact operators and vice versa.

\subsection{}
We recall the following classes of operators.

\begin{definition}\label{a-o-limited operator} 
{\em An operator 
\begin{enumerate}[a)]
\item 
$T:X\to F$ is called {\em almost Grothendieck} (shortly, $T$ is \text{\rm a-G})
if $T'$ takes disjoint $\text{\rm w}^\ast$-null sequences of $F'$ 
to \text{\rm w}-null sequences of $X'$ \cite[Def.3.1]{GM}.
\item 
$T:X\to F$ is called {\em almost limited} (shortly, $T$ is \text{\rm Lm})
if $T(B_X)$ is \text{\rm a}-limited; 
i.e., $T'$ takes disjoint $\text{\rm w}^\ast$-null sequences of $F'$ to 
norm null sequences of $X'$ \cite{EMM}. 
\item 
$T:E\to Y$ is called {\em almost Dunford--Pettis} 
(shortly, $T$ is \text{\rm a-DP}) if $T$ takes disjoint 
\text{\rm w}-null sequences to norm null ones \cite{San}.
\item 
$T:E\to Y$ is called 
{\em almost weak Dun\-ford--Pet\-tis} (shortly, $T$ is \text{\rm a-wDP}) 
if $f_n(Tx_n) \to 0$ whenever $(f_n)$ is \text{\rm w}-null in $Y'$ and 
$(x_n)$ is disjoint \text{\rm w}-null in $E$ \cite[Def.5.3.1b)]{AEG4}.
\item 
$T:E\to Y$ is called {\em $\text{\rm o}$-limited} (shortly, $T$ is \text{\rm o-Lm})
if $T[0,x]$ is limited for all 
$x\in E_+$; i.e., $(T'f_n)$ is uniformly null on all order intervals $[0,x]\subseteq E_+$
for each $\text{\rm w}^\ast$-null $(f_n)$ of $Y'$ \cite{KM}. 
\item 
$T:E\to F$ is called {\em almost $\text{\rm o}$-limited}
(shortly, $T$ is \text{\rm a-o-Lm}) if $T[0,x]$ is \text{\rm a}-limited
for all $x\in E_+$; i.e., $(T'f_n)$ is uniformly null on all order intervals $[0,x]\subseteq E_+$
for each disjoint $\text{\rm w}^\ast$-null $(f_n)$ of $F'$ \cite[Def.3.1]{KFMM}.
\end{enumerate}}
\end{definition}
\noindent
Clearly: $\text{\rm a-Lm}(X,F)\subseteq\text{\rm a-G}(X,F)$;
$\text{\rm a-DP}(E,Y)\subseteq\text{\rm a-wDP}(E,Y)$;
$\text{\rm Lm}(E,Y)\subseteq\text{\rm o-Lm}(E,Y)$; and
$\text{\rm o-Lm}(E,F)\subseteq\text{\rm a-o-Lm}(E,F)$.
\vspace{3mm}

\noindent
Let ${\cal P}\subseteq\text{\rm L}(E,F)$. We call elements of ${\cal P}$ by
${\cal P}$-operators and denote by ${\cal P}(E,F):= {\cal P}$ 
the set of all ${\cal P}$-operators in $\text{\rm L}(E,F)$.
The ${\cal P}$-operators satisfy the {\em domination 
property} if $S\in{\cal P}$ whenever $0\le S\le T\in{\cal P}$.
An operator $T\in\text{\rm L}(E,F)$ is said to be
${\cal P}$-{\em dominated} if $\pm T\le U$ for some $U\in{\cal P}$.

\subsection{Enveloping norms on spaces of regularly ${\cal P}$-operators.}
Regularly ${\cal P}$-operators were introduced in \cite{AEG3,Emel} and the 
enveloping norms in \cite{AEG4,Emel}. Here we recall basic
results. By \cite[Prop.1.3.6]{Mey}, $\text{\rm L}_r(E,F)$ 
is a Banach space under the {\em regular norm}
$\|T\|_r:=\inf\{\|S\|:\pm T\le S\in\text{\rm L}(E,F)\}$. Moreover, 
$\|T\|_r=\inf\{\|S\|: S\in\text{\rm L}(E,F), |Tx|\le S|x|\ \forall x\in E\}\ge\|T\|$
for every $T\in\text{\rm L}_r(E,F)$.
If $F$ is Dedekind complete, then $(\text{\rm L}_r(E,F),\|\cdot\|_r)$ is a Banach lattice
and $\|T\|_r=\|~|T|~\|$ for every $T\in\text{\rm L}_r(E,F)$.
The following definition was introduced in \cite[Def.2]{Emel} 
(cf. also \cite[Def.1.5.1]{AEG3}). 

\begin{definition}\label{rP-operators}{\em
Let ${\cal P}\subseteq\text{\rm L}(E,F)$. An operator $T:E\to F$ is called 
a {\it regularly} ${\cal P}$-{\it operator}
(shortly, an r-${\cal P}$-{\it operator}), if $T=T_1-T_2$ 
with $T_1,T_2\in{\cal P}\cap\text{\rm L}_+(E,F)$.
We denote by: 
${\cal P}_r(E,F)$ the set of all regular operators in ${\cal P}(E,F)$;  
and by $\text{\rm r-}{\cal P}(E,F)$ the set of all regularly 
${\cal P}$-operators in $\text{\rm L}(E,F)$.}
\end{definition}

\begin{assertion}\label{prop elem}
{\rm (\cite[Prop.1.5.2]{AEG3})}
Let ${\cal P}\subseteq\text{\rm L}(E,F)$, ${\cal P}\pm{\cal P}\subseteq{\cal P}\ne\emptyset$, 
and $T\in\text{\rm L}(E,F)$. Then the following holds.
\begin{enumerate}[\em (i)]
\item 
$T$ is an {\em r-}${\cal P}$-operator iff $T$ is a ${\cal P}$-dominated ${\cal P}$-operator.  
\item 
Suppose ${\cal P}$-operators satisfy the domination property and the modulus $|T|$ exists 
in $\text{\rm L}(E,F)$. Then $T$ is an \text{\rm r}-${\cal P}$-operator iff  $|T|\in{\cal P}$.
\end{enumerate}
\end{assertion}
\noindent
The replacement of $\text{\rm L}(E,F)$ in the definition 
of the regular norm by an arbitrary subspace 
${\cal P}\subseteq\text{\rm L}(E,F)$:
\begin{equation}\label{enveloping norm}
   \|T\|_{\text{\rm r-}{\cal P}}:=\inf\{\|S\|:\pm T\le S\in{\cal P}\} \ \ 
   \ \ (T\in\text{\rm r-}{\cal P}(E,F))
\end{equation}
gives the so-called {\em enveloping norm}
on $\text{\rm r-}{\cal P}(E,F)$ \cite{AEG4}.
Furthermore
\begin{equation}\label{Enveloping P-norm 2}
   \|T\|_{\text{\rm r-}{\cal P}}=
   \inf\{\|S\|: S\in{\cal P}\ \&\ (\forall x\in E)\ |Tx|\le S|x|\}
   \ \ \ (T\in\text{\rm r-}{\cal P}(E,F))
\end{equation}
by \cite[Lm.2.2.1]{AEG4}, and if ${\cal P}_1$ is a subspace of ${\cal P}$ then
\begin{equation}\label{Enveloping P-norm 1}
   \|T\|_{\text{\rm r-}{\cal P}_1}\ge\|T\|_{\text{\rm r-}{\cal P}}\ge\|T\|_r\ge\|T\| 
   \ \ \ \ \ (\forall\ T\in\text{\rm r-}{\cal P}_1(E,F)).
\end{equation}

\begin{assertion}\label{P-norm}
{\rm (\cite[Thm.2.3.1]{AEG4})}
Let ${\cal P}$ be a subspace of $\text{\rm L}(E,F)$
closed in the operator norm.
Then $\text{\rm r-}{\cal P}(E,F)$ is a Banach space under the
enveloping norm.
\end{assertion}
\noindent
Let ${\cal P}\subseteq\text{\rm L}(E,F)$, and 
denote ${\cal P'}:=\{T': T\in {\cal P}\}\subseteq\text{\rm L}(F',E')$. 
Clearly, $\text{\rm r-}{\cal P'}(F',E')=(\text{\rm r-}{\cal P}(E,F))'$.
Since $\|S'\|=\|S\|$, it follows from (\ref{enveloping norm})
$$
   \|T'\|_{\text{\rm r-}{\cal P'}}=\inf\{\|S'\|:\pm T'\le S'\in{\cal P'}\}=
   \inf\{\|S\|:\pm T\le S\in{\cal P}\}= \|T\|_{\text{\rm r-}{\cal P}}.
$$
If ${\cal P}\subseteq\text{\rm L}(E,F)$ is closed in the operator norm
then ${\cal P'}\subseteq\text{\rm L}(F',E')$ is also closed in the operator norm.
So, the next fact follows from Assertion \ref{P-norm}.

\begin{corollary}\label{adj-P-norm}
{\rm Let ${\cal P}$ be a subspace of $\text{\rm L}(E,F)$
closed in the operator norm. Then $\text{\rm r-}{\cal P'}(F',E')$ 
is a Banach space under the enveloping norm.}
\end{corollary}

\subsection{} 
In Section 2, we introduce the main definitions and discuss basic properties of 
affiliated operators, especially related to enveloping norms.
Section 3 is devoted to domination results for affiliated operators,
under the consideration, with special emphasize on the property \text{\rm (d)}
and on sequential \text{\rm w}-continuity of
lattice operations in Banach lattices. For further unexplained terminology and notations, 
we refer to \cite{AlBu,AEG2,AEG3,AEG4,AEW,Diestel,DU,Mey,Wnuk1,Wnuk3,Za}.

\section{Affiliated operators and enveloping norms}

Several recent papers were devoted to various 
modifications of limited, Gro\-then\-dieck, L- and M-weakly compact,
and Dunford--Pettis operators, through involving the structure of Banach lattices 
(see, e.g. \cite{AEG3,AEG4,AE,AEW,BLM1,CCJ,El,EAS,EMM,
FKM,GM,LM1,MF,KFMM,BLM0,BLM2,MFMA,Wnuk3}, Definition \ref{a-o-limited operator}). 
In this section we show that many of these operators appear as
operators affiliated to well known properties of Banach lattices
like the disjoint (dual) Schur property, the disjoint Gro\-then\-dieck property,
the property (d), and the sequential \text{\rm w}$^\ast$-continuity
of the lattice operations. In continuation of \cite{AEG4}
we shortly discuss the enveloping norms correspondent to 
these affiliated operators. 

\subsection{}
Recall that $E$ (resp. $E'$) has {\em sequentially} 
\text{\rm w}-{\em continuous} (resp. {\em sequentially} 
\text{\rm w}$^\ast$-{\em continuous}) 
{\em lattice operations} if $(|x_n|)$ is \text{\rm w}-null 
(resp. \text{\rm w}$^\ast$-null) for each \text{\rm w}-null $(x_n)$ in $E$
(resp. for each \text{\rm w}$^\ast$-null $(x_n)$ in $E'$). 

\begin{assertion}\label{KM}
{\rm (see \cite[Prop.3.1]{KM}) The following are equivalent.
\begin{enumerate}[(i)]
\item 
$E'$ has sequentially $\text{\rm w}^\ast$-continuous lattice operations.
\item
Each order interval in $E$ is limited.
\end{enumerate}}
\end{assertion}
\noindent
In particular, the dual $E'$ of each discrete Banach lattice $E$ 
with order continuous norm has sequentially \text{\rm w}$^\ast$-continuous 
lattice operations \cite[Prop.1.1]{Wnuk3}, \cite[Cor.3.2]{KM}. 
Under the disjointness assumption on a sequence 
in $E$ we have the following fact.

\begin{assertion}\label{disj w-null is mod w-null}
{\rm (cf. \cite[Thm.4.34]{AlBu})
For every disjoint $\text{\rm w}$-null $(x_n)$ in $E$, 
the sequence $(|x_n|)$ is also \text{\rm w}-null.}
\end{assertion}
\noindent
This is no longer true for \text{\rm w}$^\ast$-convergence 
(e.g. the sequence $f_n:=e_{2n}-e_{2n+1}$ is disjoint 
$\text{\rm w}^\ast$-null in $c'$ yet 
$|f_n|({\mathbb 1}_{\mathbb N})\equiv 2\not\to 0$ \cite[Ex.2.1]{CCJ}).
We recall the following properties of Banach lattices.

\begin{definition}\label{Schur property}
{\em A Banach lattice $E$ has:
\begin{enumerate}[a)]
\item 
the {\em positive Schur property} (briefly, $E\in\text{\rm (PSP)}$)
if each \text{\rm w}-null sequence in $E_+$ is norm null (cf. \cite{Wnuk1});
\item 
the {\em positive disjoint Schur property} (briefly, $E\in\text{\rm (PDSP)}$)
if each disjoint \text{\rm w}-null sequence in $E_+$ is norm null;
\item 
the {\em disjoint Schur property} (briefly, $E\in\text{\rm (DSP)}$)
if each disjoint \text{\rm w}-null sequence in $E$ is norm null;
\item 
the {\em dual positive Schur property} (briefly, $E\in\text{\rm (DPSP)}$) if each 
\text{\rm w}$^\ast$-null sequence in $E'_+$ is norm null \cite[Def.3.3]{AEW};
\item 
the {\em dual disjoint Schur property} (briefly, $E\in\text{\rm (DDSP)}$)
if each disjoint \text{\rm w}$^\ast$-null sequence in $E'$ is norm null \cite[Def.3.2]{MEM};
\item 
the {\em positive Grothendieck property} (briefly, $E\in\text{\rm (PGP)}$) if each  
\text{\rm w}$^\ast$-null sequence in $E'_+$ is \text{\rm w}-null 
(cf. \cite[p.760]{Wnuk3}); 
\item 
the {\em disjoint Grothendieck property} (briefly, $E\in\text{\rm (DGP)}$) 
if each disjoint \text{\rm w}$^\ast$-null sequence in $E'$ is \text{\rm w}-null
(cf. \cite[Def.2.1.3]{AEG3}); 
\item 
the \text{\rm (swl)}-{\em property} (briefly, $E\in\text{\rm (swl)}$)
if $(|x_n|)$ is \text{\rm w}-null for each \text{\rm w}-null sequence $(x_n)$ in $E$;
\item 
the $\text{\rm (\text{\rm sw}$^\ast$l)}$-{\em property} 
(briefly, $E\in\text{\rm (\text{\rm sw}$^\ast$l)}$)
if $(|f_n|)$ is \text{\rm w}$^\ast$-null
for each \text{\rm w}$^\ast$-null sequence $(f_n)$ in $E'$;
\item 
the {\em property} (d) (briefly, $E\in\text{\rm (d)}$) if
$(|f_n|)$ is $\text{\rm w}^\ast$-null for each disjoint 
$\text{\rm w}^\ast$-null sequence $(f_n)$ in $E'$ \cite{El,Wnuk3};
\item
the bi-{\em sequence property} (briefly, $E\in\text{\rm (bi-sP)}$)
if $f_n(x_n)\to 0$ for each \text{\rm w}$^\ast$-null 
$(f_n)$ in $E'_+$ and each disjoint $\text{\rm w}$-null $(x_n)$ in $E$ \cite[Def.3.1]{AEW};
\item 
the {\em strong} \text{\rm GP}-{\em property} (briefly, $E\in\text{\rm (s-GPP)}$)
if each almost limited subset of $E$ is relatively compact;
\item 
the {\em strong} \text{\rm BD}-{\em property} (briefly, $E\in\text{\rm (s-BDP)}$)
if each almost limited subset of $E$ is relatively \text{\rm w}-compact.
\end{enumerate}}
\end{definition}
\noindent
It is well known that $\text{\rm (PSP)}=\text{\rm (PDSP)}=\text{\rm (DSP)}$. 
Indeed, $\text{\rm (PSP)}\subseteq\text{\rm (PDSP)}$ holds trivially; 
$\text{\rm (PDSP)}\subseteq\text{\rm (DSP)}$ is due to
Assertion \ref{disj w-null is mod w-null}; and, for 
$\text{\rm (DSP)}\subseteq\text{\rm (PSP)}$ see \cite[p.16]{Wnuk1}. 
We include a short proof of the following fact.

\begin{assertion}
{\rm (\cite[Thm.4.2]{AEW}, \cite[Prop.2.4]{Wnuk3})}
{\em Let $E$ be a Banach lattice. The following are equivalent$:$
\begin{enumerate}[{\rm (i)}]
\item
$E\in\text{\rm (bi-sP)}$;
\item 
$E\in\text{\rm (Pbi-sP)}$, in the sense that if $f_n(x_n)\to 0$ for each \text{\rm w}$^\ast$-null 
$(f_n)$ in $E'_+$ and each disjoint $\text{\rm w}$-null $(x_n)$ in $E_+$.
\item 
every $\text{\rm w}^\ast$-null sequence $(f_n)$ in $E'_+$ is uniformly null on each 
dis\-joint \text{\rm w}-null $(x_n)$ in $E_+$.
\end{enumerate}}
\end{assertion}

\begin{proof}
The implication i)$\Longrightarrow$ii) is obvious, whereas
ii)$\Longrightarrow$iii) follows from 
Proposition \ref{uniform convergence on A}~i).

iii)$\Longrightarrow$i) Let $(f_n)$ be \text{\rm w}$^\ast$-null in $E'_+$
and $(x_n)$ be disjoint $\text{\rm w}$-null in $E$. 
By Assertion \ref{disj w-null is mod w-null}, $(x_n^\pm)$ are both
disjoint $\text{\rm w}$-null in $E_+$. Then $(f_n)$ is uniformly null 
on both $(x_n^\pm)$, and hence on $(x_n)=(x_n^+)-(x_n^-)$. 
By Proposition \ref{uniform convergence on A}~i), $f_n(x_n)\to 0$,
as desired.
\end{proof}

\subsection{}
Applying the redistribution between the domain and range
as in Definition \ref{Main Schur affiliated operators} to
properties of Definition \ref{Schur property}, we obtain the
correspondent affiliated operators.

\begin{definition}\label{affiliated operators EtoY} 
{\em An operator $T:E\to Y$ is called:
\begin{enumerate}[a)]
\item 
a \text{\rm [PSP]}-{\em operator} if $\|Tx_n\|\to 0$ for each
\text{\rm w}-null $(x_n)$ in $E_+$;
\item 
a \text{\rm [PDSP]}-{\em operator} if $\|Tx_n\|\to 0$ for each
disjoint \text{\rm w}-null $(x_n)$ in $E_+$;
\item 
a \text{\rm [DSP]}-{\em operator} if $\|Tx_n\|\to 0$ for each
disjoint \text{\rm w}-null $(x_n)$ in $E$;
\item 
an \text{\rm [s-GPP]}-{\em operator} if $T$ carries
almost limited subsets of $E$ onto relatively compact subsets of $Y$;
\item 
an \text{\rm [s-BDP]}-{\em operator} if $T$ carries
almost limited subsets of $E$ onto relatively \text{\rm w}-compact subsets of $Y$.
\end{enumerate}}
\end{definition}

\noindent
Clearly, 
\begin{equation}\label{s-GPP and s-BDP op1}
\text{\rm [s-GPP]}(E,Y)\subseteq\text{\rm [GPP]}(E,Y)\bigcap\text{\rm [s-BDP]}(E,Y)
\ \ \text{\rm and} 
\end{equation}
\begin{equation}\label{s-GPP and s-BDP op2}
\text{\rm [s-BDP]}(E,Y)\subseteq\text{\rm [BDP]}(E,Y).
\end{equation}
\noindent
\text{\rm [DSP]}-operators coincide with the almost Dunford--Pettis operators, 
and hence, by \cite[Thm.2.2]{AE}, 
\begin{equation}\label{PSP=PDSP=DSP}
\text{\rm [PSP]}(E,Y)=\text{\rm [PDSP]}(E,Y)=\text{\rm [DSP]}(E,Y).
\end{equation}

\begin{definition}\label{affiliated operators XtoF} 
{\em An operator $T:X\to F$ is called:
\begin{enumerate}[a)]
\item 
a \text{\rm [DPSP]}-{\em operator} if $\|T'f_n\|\to 0$ 
for each \text{\rm w}$^\ast$-null $(f_n)$ in $F_+'$;
\item 
a \text{\rm [DDSP]}-{\em operator} if $\|T'f_n\|\to 0$
for each disjoint \text{\rm w}$^\ast$-null $(f_n)$ in $F'$;
\item 
a \text{\rm [PGP]}-{\em operator} if $(T'f_n)$ is \text{\rm w}-null
for each \text{\rm w}$^\ast$-null $(f_n)$ in $F_+'$;
\item 
a \text{\rm [DGP]}-{\em operator} if $(T'f_n)$ is \text{\rm w}-null
for each disjoint \text{\rm w}$^\ast$-null $(f_n)$ in $F'$;
\item 
an \text{\rm [swl]}-{\em operator} if $(|Tx_n|)$ is \text{\rm w}-null
for each \text{\rm w}-null $(x_n)$ in $X$.
\end{enumerate}}
\end{definition}
\noindent
\text{\rm [DDSP]}-operators coincide with the almost limited
operators, whereas \text{\rm [DGP]}-operators agree with the almost Grothendieck
operators.

\begin{proposition}\label{duals 1}
$(\text{\rm [DPSP]}(X,F))'\cup(\text{\rm [DDSP]}(X,F))'\subseteq\text{\rm [PSP]}(F',X')$.
\end{proposition}

\begin{proof}
Let $(f_n)$ be disjoint \text{\rm w}-null in $F'_+$. 
Then $(f_n)$ is disjoint \text{\rm w}$^\ast$-null in $F'_+$.
If $T\in\text{\rm [DPSP]}(X,F)$ or $T\in\text{\rm [DDSP]}(X,F)$ then
in both cases $\|T'f_n\|\to 0$. Thus $T'\in\text{\rm [PDSP]}(F',X')$,
and hence $T'\in\text{\rm [PSP]}(F',X')$ by (\ref{PSP=PDSP=DSP}).
\end{proof}

\begin{definition}\label{affiliated operators EtoF} 
{\em An operator $T:E\to F$ is called:
\begin{enumerate}[a)]
\item 
a \text{\rm [dswl]}-{\em operator} if $(|Tx_n|)$ is \text{\rm w}-null
for each disjoint \text{\rm w}-null $(x_n)$;
\item 
an $\text{\rm [sw}^\ast\text{\rm l]}$-{\em operator} 
if $(|T'f_n|)$ is \text{\rm w}$^\ast$-null
for each \text{\rm w}$^\ast$-null $(f_n)$ in $F'$;
\item 
a \text{\rm [d]}-{\em operator} if $(|T'f_n|)$ is \text{\rm w}$^\ast$-null
for each disjoint \text{\rm w}$^\ast$-null $(f_n)$ in $F'$;
\item 
a \text{\rm [bi-sP]}-{\em operator} if $f_n(Tx_n)\to 0$ for each \text{\rm w}$^\ast$-null 
$(f_n)$ in $F'_+$ and each disjoint $\text{\rm w}$-null $(x_n)$ in $E$;
\item 
a \text{\rm [Pbi-sP]}-{\em operator} if $f_n(Tx_n)\to 0$ for each \text{\rm w}$^\ast$-null 
$(f_n)$ in $F'_+$ and each disjoint $\text{\rm w}$-null $(x_n)$ in $E_+$.
\end{enumerate}}
\end{definition}

\begin{proposition}\label{(d) and sw*l}
{\em For a Banach lattice $F$ the following are hold.
\begin{enumerate}[i)]
\item 
$F\in\text{\rm (d)}$ iff $\text{\rm r-[d]}(E,F)=\text{\rm L}_r(E,F)$ for every $E$.
\item 
$F'$ has sequentially $\text{\rm w}^\ast$-continuous lattice operations iff\\ 
$\text{\rm [sw}^\ast\text{\rm l]}(E,F)=\text{\rm L}_r(E,F)$ for every $E$.
\end{enumerate}}
\end{proposition}

\begin{proof}
i) For the necessity, let $E$ be a Banach lattice.
It is enough to prove $\text{\rm L}_+(E,F)\subseteq\text{\rm [d]}(E,F)$.
So, let $0\le T:E\to F$ and $(f_n)$ be disjoint $\text{\rm w}^\ast$-null
in $F'$. Since $F\in\text{\rm (d)}$ then $(|f_n|)$ is $\text{\rm w}^\ast$-null,
and then $(T'|f_n|)$ is $\text{\rm w}^\ast$-null in $E'$. It follows from
$|T'f_n|\le T'|f_n|$ that $(|T'f_n|)$ is $\text{\rm w}^\ast$-null,
and hence $T\in\text{\rm [d]}(E,F)$.
The sufficiency is immediate since $I_F\in\text{\rm [d]}(F,F)$ implies $F\in \text{\rm (d)}$.\\
ii) Just remove the disjointness condition on $(f_n)$ in the proof of i).
\end{proof}
\noindent
The next proposition shows that \text{\rm [Pbi-sP]}-operators 
agree with \text{\rm [bi-sP]}-opeators.

\begin{proposition}\label{bi-sP=Pbi-sP}
{\em $\text{\rm [bi-sP]}(E,F)=\text{\rm [Pbi-sP]}(E,F)$.}
\end{proposition}

\begin{proof}
Clearly, $\text{\rm [bi-sP]}(E,F)\subseteq\text{\rm [Pbi-sP]}(E,F)$.
Let $T\in\text{\rm [Pbi-sP]}(E,F)$, $(f_n)$ be \text{\rm w}$^\ast$-null 
in $F'_+$, and $(x_n)$ be disjoint $\text{\rm w}$-null in $E$. 
By Assertion \ref{disj w-null is mod w-null}, 
$(|x_n|)$ is disjoint \text{\rm w}-null in $E$. 
Since $T\in\text{\rm [Pbi-sP]}(E,F)$, $f_n(T|x_n|)\to 0$.
It follows from $|f_n(Tx_n)|\le f_n(T|x_n|)$ that $f_n(Tx_n)\to 0$, 
and hence $T\in\text{\rm [bi-sP]}(E,F)$.
\end{proof}

\begin{proposition}\label{(d) ve ao-lim}
{\em Let $T\in\text{\rm L}(E,F)$. The following holds.
\begin{enumerate}[{\rm i)}]
\item
$T$ is a $\text{\rm [d]}$-operator iff $T$ is almost $\text{\rm o}$-limited.
\item 
$T$ is an $\text{\rm [sw}^\ast\text{\rm l]}$-operator 
iff $T$ is $\text{\rm o}$-limited.
\end{enumerate}}
\end{proposition}

\begin{proof}
i) For the necessity, let $T\in\text{\rm [d]}(E,F)$. Suppose
$x\in E_+$ and $(f_n)$ is disjoint $\text{\rm w}^\ast$-null in $F'$.
By the assumption, $(|T'f_n|)$ is \text{\rm w}$^\ast$-null,
and hence $|T'f_n|x\to 0$. By the Riesz--Kantorovich formula,
$|T'f_n|x=\sup\{|(T'f_n)y|: |y|\le x\}\to 0$, and hence
$(T'f_n)$ is uniformly null on each $[0,x]$. 
Thus $T\in\text{\rm a-o-Lm}(E,F)$.

For the sufficiency, let $T\in\text{\rm a-o-Lm}(E,F)$.
Suppose $(f_n)$ is disjoint $\text{\rm w}^\ast$-null in $F'$.
In order to prove $T\in\text{[d]}(E,F)$, we need 
to show that $(|T'f_n|)\stackrel{\text{\rm w}^\ast}{\to}0$.
It is enough to prove that $|T'f_n|x\to 0$ for each $x\in E_+$.
Let $x\in E_+$. By the assumption, $\sup\{|(T'f_n)y|: |y|\le x\}\to 0$.
Therefore, the Riesz--Kantorovich formula implies
$|T'f_n|x\to 0$, and hence $T\in\text{[d]}(E,F)$.

ii) Just remove the disjointness condition on $(f_n)$ in the proof of i).
\end{proof}

\subsection{}
Affiliated operators from the previous subsection form
vector spaces, which are complete under the operator norm;
the details are included in the next lemma.

\begin{lemma}\label{complete in the operator norm}
{\em The following sets of affiliated operators are vector spaces 
which are complete in the operator norm.
\begin{enumerate}[{\rm i)}]
\item 
$\text{\rm [PSP]}(E,Y)$.
\item 
$\text{\rm [DPSP]}(X,F)$ and $\text{\rm [DDSP]}(X,F)$.
\item
$\text{\rm [PGP]}(X,F)$ and $\text{\rm [DGP]}(X,F)$.
\item
$\text{\rm [swl]}(X,F)$ and $\text{\rm [dswl]}(E,F)$.
\item
$\text{\rm [sw}^\ast\text{\rm l]}(E,F)$ and $\text{\rm [d]}(E,F)$.
\item
$\text{\rm [bi-sP]}(E,F)$.
\item
$\text{\rm [GPP]}(X,Y)$ and $\text{\rm [s-GPP]}(E,Y)$. 
\item
$\text{\rm [BDP]}(X,Y)$ and $\text{\rm [s-BDP]}(E,Y)$.
\end{enumerate}}
\end{lemma}

\begin{proof}
We skip trivial checking that all sets of affiliated operators in the lemma
are vector spaces. It remains to show that
each space of affiliated operators under the consideration is a closed
in the operator norm subspace of the correspondent space of all linear operators.
As arguments here are straightforward and standard, we present them in the basic cases.
\begin{enumerate}[{\rm i)}]

\item
Let $\text{\rm [PSP]}(E,Y)\ni T_k\stackrel{\|\cdot\|}{\to}T\in\text{\rm L}(E,Y)$. 
Let $(x_n)$ be $\text{\rm w}$-null in $E_+$. We need to show $\|Tx_n\|\to 0$. 
Let $\varepsilon>0$. Pick some $k\in\mathbb{N}$ with $\|T-T_k\|\le\varepsilon$. 
Since $T_k\in\text{\rm PSP}(E,Y)$, there exists $n_0$ such that
$\|T_k x_n\|\le\varepsilon$ for $n\ge n_0$. Take $M\in\mathbb{R}$ 
satisfying $\|x_n\|\le M$ for all $n\in\mathbb{N}$. Since
$$
   \|Tx_n\| = \|(T -T_k)x_n+T_k x_n\| \le
   \|T -T_k\|\cdot \|x_n\|+\|T_kx_n\|\le \varepsilon(M+1)
$$
for $n\ge n_0$, and since $\varepsilon>0$ is arbitrary, $\|Tx_n\|\to 0$.

\item
As the case of \text{\rm [DDSP]}(X,F) is similar, we confine ourselves to considering
\text{\rm [DPSP]}(X,F).

Let $\text{\rm [DPSP]}(X,F)\ni T_k\stackrel{\|\cdot\|}{\to}T\in\text{\rm L}(X,F)$,
and let $(f_n)$ be $\text{\rm w}^\ast$-null in $F_+'$.
In order to show $(T'f_n)$ is norm null, let $\varepsilon>0$
and pick $k$ with $\|T'-T'_k\|\le\varepsilon$. Since
$T_k\in\text{\rm [DPSP]}(X,F)$, there exists $n_0$ with
$\|T_k'f_n\|\le\varepsilon$ for all $n\ge n_0$.
As $(f_n)$ is $\text{\rm w}^\ast$-null, there exists $M\in\mathbb{R}$
satisfying $\|f_n\|\le M$ for all $n\in\mathbb{N}$. Since 
$$
   \|T'f_n\|\le\|T'f_n-T'_kf_n\|+\|T'_kf_n\|\le
   \|T'_k-T'\|\|f_n\|+\varepsilon\le\varepsilon(M+1)
$$
for $n\ge n_0$. It follows $\|T'f_n\|\to 0$, as desired.

\item 
As the case of \text{\rm [DGP]}(X,F) is similar, we consider 
\text{\rm [PGP]}(X,F) only.

Let $\text{\rm [PGP]}(X,F)\ni T_k\stackrel{\|\cdot\|}{\to}T\in\text{\rm L}(X,F)$,
and let $(f_n)$ be $\text{\rm w}^\ast$-null in $F_+'$.
In order to show that $(T'f_n)$ is \text{\rm w}-null,
pick a $g\in F''$, and let $\varepsilon>0$.
Fix any $k$ with $\|T'-T'_k\|\le\varepsilon$. Since
$T_k\in\text{\rm [PGP]}(X,F)$, there exists $n_0$
with $|g(T_k'f_n)|\le\varepsilon$ for all $n\ge n_0$.
Let $M\in\mathbb{R}$ be such $\|f_n\|\le M$ 
for all $n\in\mathbb{N}$. Because of 
$$
   |g(T'f_n)|\le|g(T'f_n-T_k'f_n)|+|g(T_k'f_n)|\le
$$
$$
   \|g\|\|T'-T_k'\|\|f_n\|+\varepsilon\le(\|g\|M+1)\varepsilon
$$
for $n\ge n_0$, and since $\varepsilon>0$ is arbitrary,
it follows $g(T'f_n)\to 0$. Since $g\in F''$ is arbitrary,
$T\in\text{\rm [PGP]}(X,F)$.

\item 
We consider \text{\rm [swl]}(X,F) only. 
The case of \text{\rm [dswl]}(E,F) is similar.

Let $\text{\rm [swl]}(X,F)\ni T_k\stackrel{\|\cdot\|}{\to}T\in\text{\rm L}(X,F)$ 
and let $(x_n)$ be  $\text{\rm w}$-null in $X$. We need to show 
$|Tx_n|\stackrel{\text{\rm w}}{\to} 0$ in $F$. 
Let $f\in F'$. There exists an $M\in{\mathbb R}$
with $\|x_n\|\le M$ for all $n\in{\mathbb N}$. Take some $\varepsilon>0$ 
and pick  $k\in{\mathbb N}$ with $\|T-T_k\|\le\varepsilon$. 
Choose $n_0$ such that $|f|(|T_kx_n|)\le\varepsilon$ for all $n\ge n_0$. Then
$$
   |f(|Tx_n|)|\le|f|(|(T-T_k)x_n +T_k x_n|)\le
$$
$$
   \|f\|\cdot\|T-T_k\|\cdot M+|f|(|T_k x_n|)\le
   \varepsilon(\|f\|M+1).
$$
Since $\varepsilon>0$ is arbitrary, $f(|Tx_n|)\to 0$; and, since $f\in F'$ is arbitrary,
$|Tx_n|\stackrel{\text{\rm w}}{\to} 0$.

\item
We consider $\text{\rm [d]}(E,F)$ only. 
The case of $\text{\rm [sw}^\ast\text{\rm l]}(E,F)$ is similar.

Let $\text{\rm [d]}(E,F)\ni T_k\stackrel{\|\cdot\|}{\to}T\in\text{\rm L}(E,F)$, 
and let $(f_n)$ be disjoint \text{\rm w}$^\ast$-null in $F'$.
We need to show $(|T'f_n|)\stackrel{\text{\rm w}^\ast}{\to}0$.
It is enough to prove that $|T'f_n|x\to 0$ for each $x\in E_+$.
Let $x\in E_+$ and $\varepsilon>0$. 
Pick $k\in\mathbb{N}$ with $\|T'-T_k'\|\le\varepsilon$.
By the assumption, $|T_k'f_n|x\to 0$. 
So, let $n_0\in\mathbb{N}$ be such that
$|T_k'f_n|x\le\varepsilon$ whenever $n\ge n_0$.
As $(f_n)$ is $\text{\rm w}^\ast$-null, 
there exists $M\in\mathbb{R}$ with $\|f_n\|\le M$ for all $n\in\mathbb{N}$.
By the Riesz--Kantorovich formula, for $n\ge n_0$,
$$
   |T'f_n|x=\sup\{|(T'f_n)y|: |y|\le x\}\le
$$
$$ 
   \sup\{|((T'-T_k')f_n)y|: |y|\le x\}+\sup\{|(T_k'f_n)y|: |y|\le x\}\le
$$
$$
   \sup\{\|T'-T_k'\|\cdot\|f_n\|\cdot\|y\|: |y|\le x\}+|T_k'f_n|x\le
   \varepsilon(M\|x\|+1).
$$
Therefore $|T'f_n|x\to 0$, and hence $T\in\text{[d]}(E,F)$.

\item
Let $\text{\rm [bi-sP]}(E,F)\ni T_k\stackrel{\|\cdot\|}{\to}T\in\text{\rm L}(E,F)$.
Let $(f_n)$ be \text{\rm w}$^\ast$-null in $F'_+$ and let 
$(x_n)$ be disjoint $\text{\rm w}$-null in $E$.
We need to show $f_n(Tx_n)\to 0$. Pick $M\in{\mathbb R}$ such that
$\|f_n\|\le M$ and $\|x_n\|\le M$ for all $n\in{\mathbb N}$.
Take some $\varepsilon>0$. Pick $k\in{\mathbb N}$ with
$\|T-T_k\|\le\varepsilon$. Since $T_k\in\text{\rm [bi-sP]}(E,F)$,
there exists $n_0\in{\mathbb N}$ such that $|f_n(T_kx_n)|\le\varepsilon$
for $n\ge n_0$. Then
$$
   |f_n(Tx_n)|\le|f_n((T-T_k)x_n)|+|f_n(T_kx_n)|\le 
$$
$$   
   \|f_n\|\cdot\|T-T_k\|\cdot\|x_n\|+\varepsilon\le (M^2+1)\varepsilon 
   \ \ \ (\forall n\ge n_0). 
$$
Since $\varepsilon>0$ is arbitrary, $f_n(Tx_n)\to 0$.

\item
As the case of $\text{\rm [GPP]}(X,Y)$ is similar, 
we consider $\text{\rm [s-GPP]}(E,Y)$ only.

Let $\text{\rm [s-GPP]}(E,Y)\ni T_k\stackrel{\|\cdot\|}{\to}T\in\text{\rm L}(E,Y)$,
and let $A\subseteq E$ be \text{\rm a}-limited. We need to show that $T(A)$ 
is relatively compact. Since \text{\rm a}-limited sets are bounded, 
there exists $M\in{\mathbb R}$ with $\|x\|\le M$ for all $x\in A$. 
Choose $\varepsilon>0$ and pick a $k\in{\mathbb N}$ such that
$\|T-T_k\|\le\varepsilon$. Then
$$
   Tx=T_kx+(T-T_k)x\in T_k(A)+\|T-T_k\|\cdot\|x\|B_Y=
   T_k(A)+\varepsilon M\cdot B_Y
$$ 
for all $x\in A$, and hence $T(A)\subseteq T_k(A)+\varepsilon M\cdot B_Y$. 
By the assumption, $T_k(A)$ is relatively compact.
Since $\varepsilon>0$ is arbitrary, $T(A)$ is totally bounded 
and hence is relatively compact, as desired.

\item
As the case of $\text{\rm [BDP]}(X,Y)$ is similar, 
we consider $\text{\rm [s-BDP]}(E,Y)$ only.

Let $\text{\rm [s-BDP]}(E,Y)\ni T_k\stackrel{\|\cdot\|}{\to}T\in\text{\rm L}(E,Y)$,
and let $A\subseteq E$ be \text{\rm a}-limited. 
We need to show that $T(A)$ is relatively \text{\rm w}-compact.
Since \text{\rm a}-limited sets
are bounded, there exists $M\in{\mathbb R}$ such that
$\|x\|\le M$ for all $x\in A$. Take $\varepsilon>0$ and pick any $k\in{\mathbb N}$ with
$\|T-T_k\|\le\varepsilon$. Then $T(A)\subseteq T_k(A)+\varepsilon M\cdot B_Y$, 
as above in vii).
By the assumption, $T_k(A)$ is relatively \text{\rm w}-compact.
Since $\varepsilon>0$ is arbitrary, $T(A)$ is relatively \text{\rm w}-compact
by the Grothendieck result \cite[Thm.3.44]{AlBu}.
\end{enumerate}
\end{proof}
\noindent
The next result follows from Theorem \ref{P-norm} and 
Lemma \ref{complete in the operator norm}.

\begin{theorem}\label{aff op vs envelop} 
{\rm 
Let $E$ ad $F$ be a Banach lattices. Then
$\text{\rm r-[PSP]}(E,F)$, $\text{\rm r-[DPSP]}(E,F)$, $\text{\rm r-[DDSP]}(E,F)$,
$\text{\rm r-[PGP]}(E,F)$, $\text{\rm r-[DGP]}(E,F)$,
$\text{\rm r-[swl]}(E,F)$, $\text{\rm r-[dswl]}(E,F)$,
$\text{\rm r-[sw}^\ast\text{\rm l]}(E,F)$, $\text{\rm r-[d]}(E,F)$,
$\text{\rm r-[bi-sP]}(E,F)$, $\text{\rm r-[GPP]}(E,F)$,\\
$\text{\rm r-[s-GPP]}(E,F)$, $\text{\rm r-[BDP]}(E,F)$, 
and $\text{\rm r-[s-BDP]}(E,F)$ are all Banach spaces, 
each under its own enveloping norm.} 
\end{theorem}

\section{Domination for affiliated operators} 

Here we gather domination results for defined above affiliated operators.
Some of them already appeared in the literature, the others seem new. 

\subsection{}
The \text{\rm [s-GPP]-}operators do not satisfy the domination property in the 
strong sense that even an operator which is dominated by a rank one operator
need not to be a \text{\rm [GPP]}-operator.

\begin{example}\label{[s-GPP] vs rank one}
{\em (cf. \cite[Ex.5.30]{AlBu})
Define operators $T,S: L^1[0,1]\to\ell^\infty$ by 
$T(f):=(\int_0^1 f(t)dt)_{k=1}^\infty$, and
$S(f):=(\int_0^1 f(t)r_k^+(t)dt)_{k=1}^\infty$,
where $r_k$ are the Rademacher functions on $[0,1]$. 
Then $T$ is a rank one operator, and hence
$T\in\text{\rm [s-GPP]}(L^1[0,1],\ell^\infty)$.
Moreover, $0\le S\le T$, yet $S$ is not 
a \text{\rm [GPP]}-operator.
To see this, consider the sequence of the Rademacher functions
$(r_n)$ in $[0,\mathbb{1}]\subseteq L^1[0,1]$, which is 
an \text{\rm a}-limited subset of $L^1[0,1]$, e.g.
by Proposition \ref{if (d) then o-intervals are a-limited}.
The sequence $(Sr_n)=(\frac{1}{2}e_n)$, where $e_n$ are
the n-th unite vectors in $\ell^\infty$, has no norm convergent
subsequences, and hence 
$S\not\in\text{\rm [GPP]}(L^1[0,1],\ell^\infty)$.}
\end{example} 
\noindent
We do not know whether or not the operator $S$ in 
Example \ref{[s-GPP] vs rank one} is a \text{\rm [BDP]}-operator.

\subsection{}
It turn out that the property \text{\rm (d)} and 
the sequential \text{\rm w}-continuity ($\text{\rm w}^\ast$-continuity)
of lattice operations play an important role for the domination property. 
Firstly, we include some related elementary facts. 

\begin{proposition}\label{if (d) then o-intervals are a-limited}
The following are equivalent.
\begin{enumerate}[{\em i)}]
\item 
$E\in\text{\rm (d)}$.
\item
Each order interval in $E$ is \text{\rm a}-limited.
\end{enumerate}
\end{proposition} 

\begin{proof}
i)$\Longrightarrow$ii) 
It suffices to show that intervals $[-a,a]$ are \text{\rm a}-limited.
Let $a\in E_+$, and let $(f_n)$ be disjoint \text{\rm w}$^\ast$-null 
in $E'$. We need to show that $(f_n)$ is uniformly null on $[-a,a]$.
By Assertion \ref{uniform convergence on A}, it is enough to show 
that $f_n(a_n)\to 0$ for each sequence $(a_n)$ in $[-a,a]$.
So, let $(a_n)$ be in $[-a,a]$. Since $E\in\text{\rm (d)}$ then
$(|f_n|)$ is \text{\rm w}$^\ast$-null in $E'_+$, and hence 
$f_n(a)\to 0$. It follows from $-f_n(a)\le f_n(a_n)\le f_n(a)$ 
for all $n\in\mathbb{N}$ that $f_n(a_n)\to 0$. 
By Assertion \ref{uniform convergence on A}, 
$(f_n)$ is uniformly null on $[-a,a]$, as desired.

ii)$\Longrightarrow$i) 
Let $(f_n)$ be disjoint \text{\rm w}$^\ast$-null in $E'$.
We need to show that $(|f_n|)$ is \text{\rm w}$^\ast$-null.
Pick an $a\in E_+$. By the assumption, $(f_n)$ is uniformly 
null on $[-a,a]$, and in view of the Riesz--Kantorovich formula, 
$|f_n|a=\sup\limits_{y\in[-a,a]}|f_n(y)|\to 0$. Since $a\in E_+$ 
is arbitrary, $(|f_n|)$ is \text{\rm w}$^\ast$-null, as desired.
\end{proof}
\noindent 
The proof of the following result of \cite{KM} consists in 
removing the disjointness condition in the proof of
Proposition \ref{if (d) then o-intervals are a-limited}.

\begin{assertion}\label{KM}
{\rm The following are equivalent.
\begin{enumerate}[(i)]
\item 
$E'$ has sequentially $\text{\rm w}^\ast$-continuous lattice operations.
\item
Each order interval in $E$ is limited.
\end{enumerate}}
\end{assertion}

\noindent
Here we gather several (partially positive) domination results. 

\begin{theorem}\label{main dominated}
{\em Let $E$ and $F$ be Banach lattices. The following spaces of operators
satisfy the domination property.
\begin{enumerate}[{\rm i)}]
\item 
$\text{\rm [PSP]}(E,F)$. 
\item 
$\text{\rm [DPSP]}(E,F)$.
\item
$\text{\rm [DDSP]}(E,F)$, under the assumption $F\in\text{\rm (d)}$.
\item
$\text{\rm [PGP]}(E,F)$. 
\item
$\text{\rm [DGP]}(E,F)$, under the assumption $F\in\text{\rm (d)}$.
\item 
$\text{\rm [dswl]}(E,F)$.
\item 
$\text{\rm [swl]}(E,F)$, under the assumption that
$E$ has sequentially \text{\rm w}-continuous lattice operations.
\item
$\text{\rm [sw}^\ast\text{\rm l]}(E,F)$,
under the assumption that
$F'$ has sequentially $\text{\rm w}^\ast$-continuous lattice operations.
\item
$\text{\rm [d]}(E,F)$, under the assumption $F\in\text{\rm (d)}$.
\item
$\text{\rm [bi-sP]}(E,F)$.
\end{enumerate}}
\end{theorem}

\begin{proof}
As above, we restrict ourselves to basic cases.
\begin{enumerate}[{\rm i)}]
\item
Let $0\le S\le T\in\text{\rm [PSP]}(E,F)$ and 
let $(x_n)$ be \text{\rm w}-null in $E_+$. 
Since $T\in\text{\rm [PSP]}(E,F)$
then $\|Tx_n\|\to 0$. It follows from
$0\le Sx_n\le Tx_n$ that $\|Sx_n\|\to 0$,
and hence $S\in\text{\rm [PSP]}(E,F)$.
\item
Let $0\le S\le T\in\text{\rm [DPSP]}(E,F)$ and 
let $(f_n)$ be \text{\rm w}$^\ast$-null in $F'_+$.
Since $T\in\text{\rm [DPSP]}(E,F)$
then $\|T'f_n\|\to 0$. It follows from 
$0\le S'\le T'$ that $0\le S'f_n\le T'f_n$,
and hence $\|S'f_n\|\to 0$. Thus, $S\in\text{\rm [DPSP]}(E,F)$.
\item
As \text{\rm [DDSP]}-operators agree with almost limited operators, 
we refer for the proof to \cite[Cor.3]{El}.
\item
Let $0\le S\le T\in\text{\rm [PGP]}(E,F)$, 
and $(f_n)$ be \text{\rm w}$^\ast$-null in $F_+'$.
In order to prove $S\in\text{\rm [PGP]}(E,F)$, it suffices
to prove $g(S'f_n)\to 0$ for all $g\in E'_+$. Let $g\in E'_+$. 
Since $T\in\text{\rm [PGP]}(E,F)$, $g(T'f_n)\to 0$.
It follows from $0\le g(S'f_n)\le g(T'f_n)$ that
$g(S'f_n)\to 0$, as desired.
\item
As \text{\rm [DGP]}-operators agree with almost Grothendieck
operators, we refer for the proof to \cite[Prop.3.7]{GM}.
\item
Let $0\le S\le T\in\text{\rm [dswl]}(E,F)$, and 
let $(x_n)$ be disjoint \text{\rm w}-null in $E$.
In order to prove $S\in\text{\rm [dswl]}(E,F)$, it suffices
to prove $f(|Sx_n|)\to 0$ for all $f\in F'_+$. So, let $f\in F'_+$.
By Assertion \ref{disj w-null is mod w-null}, $(|x_n|)$ is \text{\rm w}-null.
Since $T\in\text{\rm [dswl]}(E,F)$ then 
$(T|x_n|)=(|T(|x_n|)|)$ is \text{\rm w}-null, and hence
$f(T|x_n|)\to 0$. It follows from
$|Sx_n|\le S|x_n|\le T|x_n|$ that $f(|Sx_n|)\to 0$ as desired.
\item
Let $0\le S\le T\in\text{\rm [swl]}(E,F)$, 
and $(x_n)$ be \text{\rm w}-null in $E$.
It suffices to prove $f(|Sx_n|)\to 0$ for all $f\in F'_+$.
Let $f\in F'_+$. By the assumption, $(|x_n|)$ is 
\text{\rm w}-null. Since $T\in\text{\rm [swl]}(E,F)$, 
$f(T|x_n|)=f(|Tx_n|)\to 0$. In view of $|Sx_n|\le S|x_n|\le T|x_n|$,
$f(|Sx_n|)\to 0$, and hence $S\in\text{\rm [swl]}(E,F)$.
\item
It follows from Proposition \ref{(d) and sw*l} ii).
\item 
It follows from Proposition \ref{(d) and sw*l} i).
\item
Let $0\le S\le T\in\text{\rm [bi-sP]}(E,F)$.
Let $(f_n)$ be \text{\rm w}$^\ast$-null in $F'_+$, and
let $(x_n)$ be disjoint \text{\rm w}-null in $E$.
In order to prove $S\in\text{\rm [bi-sP]}(E,F)$, it suffices
to prove $f_n(Sx_n)\to 0$. By Assertion \ref{disj w-null is mod w-null}, 
$(|x_n|)$ is disjoint \text{\rm w}-null in $E$, 
and, since $T\in\text{\rm [bi-sP]}(E,F)$, then $f_n(T|x_n|)\to 0$. 
It follows from $|f_n(Sx_n)|\le f_n(S|x_n|)\le f_n(T|x_n|)$ 
that $f_n(Sx_n)\to 0$, and hence $S\in\text{\rm [bi-sP]}(E,F)$. 
\end{enumerate}
\end{proof}
\noindent
In view of \cite[Thm.4.31]{AlBu}, the next fact follows 
form Theorem~\ref{main dominated}~vii).

\begin{corollary}\label{swl dominated property if E is AM}
Let $E$ be an \text{\rm AM}-space, and let $0\le S\le T\in\text{\rm [swl]}(E,F)$. 
Then $S\in\text{\rm [swl]}(E,F)$.
\end{corollary}

{\tiny 
}

\end{document}